\documentclass{amsart}

\usepackage{amssymb}
\usepackage{amsmath}
\usepackage{cite}
\usepackage[small,nohug,heads=vee]{diagrams}
\usepackage{algorithmic}
\diagramstyle[labelstyle=\scriptstyle]

\title{Algebraic Equations in State Condition}
\author{Cheolgyu Lee}
\address[CL]{Center for Geometry and Physics, Institute for Basic Science (IBS), Pohang 37673, Republic of Korea}
\address[CL]{Department of Mathematics, POSTECH,
77 Cheongam-ro, Nam-gu, Pohang, Gyeongbuk, 37673, Korea.
}
\email{ghost279.math@gmail.com}
\thanks{This work was supported by IBS-R003-D1. The author was partially supported by the following grants funded by
the government of Korea: NRF grant 2011-0030044 (SRC-GAIA) and
NRF-2013R1A1A2010649.
\\It is great pleasure to thank Donghoon Hyeon, who introduced the author to the original statement and encouraged him. The author also wants to thank Junyoung Park, who pointed out that the original statement is not true.
}
\date{\today}

\newtheorem{theorem}{Theorem}[section]
\newtheorem{lemma}[theorem]{Lemma}

\newtheorem{corollary}[theorem]{Corollary}

\newenvironment{definition}[1][Definition]{\begin{trivlist}
\item[\hskip \labelsep {\bfseries #1}]}{\end{trivlist}}

\begin{document}
\begin{abstract}
In this paper, we will prove that a problem deciding whether there is an upper-triangular coordinate in which a character is not in the state of a Hilbert point is NP-hard. This problem is related to the GIT-semistability of a Hilbert point.
\end{abstract}
\maketitle
\section{Introduction}
Let $k$ be an algebraically closed field of characteristic zero and $^{r}S=k[x_{1}, \ldots, x_{r}]$ be a polynomial ring of $r$ variables graded by degree. Let's omit the superscript $^{r}$ if there is no confusion. When non-negative integers $d$ and $b$ are fixed, there is a projective space
\begin{displaymath}
E_{d, b}^{r}=\mathbb{P}\bigg(\bigwedge^{b} \textrm{}^{r}S_{d}\bigg),
\end{displaymath}
 which is a $\textup{GL}_{r}(k)$-representation. Let $T_{r}$ be the maximal torus of $\textup{GL}_{r}(k)$ which consists of diagonal matrices and $U_{r}$ be the set of all upper-triangular matrices with $1$'s in the diagonal. There is a $G$-equivariant closed immersion
\begin{displaymath}
i_{r, P, d}:\textup{Hilb}^{P}(\mathbb{P}^{r-1}_{k})\rightarrow E_{d, Q(d)}^{r}
\end{displaymath}
 for $d\geq g_{P}$ where $g_{P}$ is the Gotzmann number associated to a Hilbert polynomial $P$, which is defined in \cite{Gotzmann}. Also $Q(d)=\binom{r+d-1}{d}-P(d)$.

For any point $v\in E_{d, b}^{r}$, the collection of states $\Xi_{G.v}=\{\Xi_{g.v}(T)|g\in\textup{GL}_{r}(k)\}$ (defined in \cite{Kempf}) of $v$ determines whether $v$ is semistable or not, as stated in \cite{Ian}. If $v$ is unstable, $\Xi_{G.v}$ determines the Hesselink strata of $\mathbb{P}\bigg(\bigwedge^{b} S_{d}\bigg)^{\textup{us}}$ that contains $v$, which is stated in \cite{Hesselink}.
For an arbitrary character $\chi$ of $T$, $Z_{v, \chi}=\{g\in \textup{GL}_{r}(k)|\chi\notin \Xi_{g.v}\}$ is a Zariski-closed subset of $\textup{GL}_{r}(k)$. In this paper, we will construct a solvability check problem (\textbf{SC}) which is equivalent to deciding if an arbitrary system of algebraic equations is solvable (\textbf{SysAl}) by \textit{specializing} the defining equation of some $Z_{v, \chi}$ to the defining equation of $U_{r}\cap Z_{v, \chi}$ in $U_{r}$.
 
It's a well-known fact that to decide whether an arbitrary system of algebraic equation is  solvable is an NP-hard problem.(\cite{Mar}) We will show that this problem can be reduced to the problem asking whether there is a $g\in U_{r}$ such that $\chi\notin \Xi_{g.w}$, in polynomial time. This means that such a problem is NP-hard. This problem is related to the GIT-semistability of a Hilbert point. By solving finitely many such problems, we can decide whether a Hilbert point is semistable or not.
\section{Definitions and notations}
First of all, we need to define the notion of generalization of a system of algebraic equation.
\begin{definition}
Suppose $I$ be an ideal of $S=k[x_{1}, \ldots, x_{r}]$. An ideal $J$ of a finitely generated $k$ algebra $R$ is a {\it generalization} of $I$ under $\pi$ if there is a surjective ring homomorphism $\pi:R\rightarrow S$ and a minimal generator $\{z_{1}, \ldots, z_{r'}\}$ of $R$ satisfying follows:
\begin{itemize}
\item For any $1\leq i\leq r'$, $\pi(z_{i})\in k\cup\{x_{1}, \ldots, x_{r}\}$
\item $\pi(J)=I$
\end{itemize}
$I$ is a {\it specialization} of $J$ if $J$ is a generalization of $I$.
\end{definition}

For example, $I=\langle x^{2}+y^{2}\rangle \subset k[x, y]$ is a specialization of $J=\langle z(x^{2}+y^{2}), zw\rangle\subset k[x, y, z, w]$ under the map $\pi:k[x, y, z, w]\rightarrow k[x, y]$ which satisfying $\pi(x)=x$, $\pi(y)=y$, $\pi(z)=1$ and $\pi(w)=0$.

We define some notation. Let $<_{\textup{lex}}$ be a lexicographic monomial order satisfying $x_{i+1}<_{\textup{lex}}x_{i}$ and let $A_{d, b}^{r}=\bigwedge^{b} \textrm{}^{r}S_{d}$.
Let $^{r}M_{d}$ be the set of all monomials in $^{r}S_{d}$ and
\begin{displaymath}
W_{d, b}^{r}=\bigg\lbrace\bigwedge_{i=1}^{b} m_{i}\bigg\vert m_{i}\in\textrm{ }^{r}M_{d},\textrm{ } m_{i}>_{\textup{lex}} m_{i+1}\bigg\rbrace.
\end{displaymath}
$W_{d, b}^{r}$ is a basis of $A_{d, b}^{r}$. Suppose $v\in A_{d, b}^{r}$ and $w\in W_{d, b}^{r}$. We define $v_{w}$ to be the $w$-component of the vector $v$. That is,
\begin{displaymath}
v=\sum_{w\in W_{d, b}^{r}} v_{w}w.
\end{displaymath}
Let $[v]\in E_{d, b}^{r}$ be the line in $A_{d, b}^{r}$ through $v$ and the origin of $A_{d, b}^{r}$. For any $g\in \textup{GL}_{r}(k)$, $g_{ij}\in k$ is the component of $g$ in the $i$'th row and $j$'th column. That is,
\begin{displaymath}
g=\left[\begin{array}{cccc}
g_{11} & \cdots & g_{1j} & \cdots \\
\vdots & \ddots & \vdots & \cdots \\
g_{i1} & \cdots & g_{ij} & \cdots \\
\vdots & \vdots & \vdots & \ddots
\end{array}\right]
\end{displaymath}
Also, $\textup{GL}_{r}(k)$ action on $^{r}S$ is given by $g.x_{i}=\sum_{1\leq j\leq r} g_{ji}x_{j}$. Note that this action is a left action on $^{r}S$.
For any $v\in A_{d, b}^{r}$ and $w\in k[W_{d, b}^{r}]$, $(g.v)_{w}$ means $((\textup{id}_{\Gamma(\textup{GL}_{r}(k), \mathcal{O}_{\textup{GL}_{r}(k)})}\otimes_{k}e_{v})\circ\phi)(w)$ when $g$ is an indeterminate. Here $\phi$ is the co-action map
\begin{displaymath}
\phi : k[W_{d, b}^{r}]\rightarrow \Gamma(\textup{GL}_{r}(k), \mathcal{O}_{\textup{GL}_{r}(k)}) \otimes_{k} k[W_{d, b}^{r}]\cong k[\{g_{ij}\}_{i, j=1}^{r}]_{\det g}\otimes_{k} k[W_{d, b}^{r}]
\end{displaymath}
and $e_{v}$ is the evaluation map $e_{v}:k[W_{d, b}^{r}]\rightarrow k$ at $v$.
Let's define $\chi_{i}\in X(T_{r})$ for all $1\leq i\leq r$ as follows:
\begin{displaymath}
\chi_{i}(D)=D_{ii}
\end{displaymath}
where $D\in T_{r}$. Let $\xi_{d, b}^{r}=\frac{db}{r}(1, \ldots, 1)\in X(T_{r})_{\mathbb{R}}=X(T_{r})\otimes_{\mathbb{Z}}\mathbb{R}$. Here $X(T_{r})$ is the group of characters of the algebraic torus $T_{r}$.

Let $L_{r}=\{g\in\textup{GL}_{r}(k)|g\textrm{ is lower-triangular.}\}$. Let's define a specialization map $\theta_{r}:\Gamma(\textup{GL}_{r}(k), \mathcal{O}_{\textup{GL}_{r}(k)})\rightarrow \Gamma(U_{r}, \mathcal{O}_{U_{r}})$ as follows.
\begin{displaymath}
\theta_{r}(z_{ij})=\left\{
\begin{array}{ll} 
1 & \textrm{if } i=j\\
0 & \textrm{if } i>j\\
z_{ij} & \textrm{if } i<j
\end{array}\right.
\end{displaymath}
where $\Gamma(\textup{GL}_{r}(k), \mathcal{O}_{\textup{GL}_{r}(k)})=k[\{z_{ij}\}_{i, j=1}^{r}]_{\textup{det}z}$ and $\Gamma(U_{r}, \mathcal{O}_{U_{r}})=k[\{z_{ij}\}_{1\leq i< j\leq r}]$. For any $C\subset k[\{z_{ij}\}_{i, j=1}^{r}]_{\textup{det}z}$, let $\textup{span} C$ be the $k$-subspace of $k[\{z_{ij}\}_{i, j=1}^{r}]_{\det z}$ spanned by $C$. Let $\Sigma_{r}$ be the permutation group on the set $\{1, 2, \ldots, r\}$, which is a subgroup of $\textup{GL}_{r}(k)$. Let $\Delta_{v}$ be the convex hull of $\Xi_{v}$ in $X(T_{r})_{\mathbb{R}}$ for all $v\in E_{d, b}^{r}$.
\section{Polynomial coefficients in some special cases}
 Suppose $v\in A_{d, b}^{r}$. In this section, we will compute $v_{w}$ for some special $w\in W_{d, b}^{r}$. Let's compute it when $b=1$ first.
\begin{lemma}
\label{onewedge}
Suppose $r\geq 2$. Let $p\in \textrm{}^{r}S_{d}=A_{d, 1}^{r}$. For any $g\in \textup{GL}_{r}(k)$, 
\begin{displaymath}
(g.p)_{x_{1}^{d-j}x_{2}^{j}}=\sum_{i_{1}+\ldots+i_{r}=j} \frac{\prod_{1\leq a\leq r}g_{2a}^{i_{a}}}{\prod_{1\leq a\leq r} i_{a}!}\frac{\partial^{j} p}{\partial x_{1}^{i_{1}}\ldots \partial x_{r}^{i_{r}}}\bigg\vert_{x_{i}=g_{1i}}
\end{displaymath}
\end{lemma}
\begin{proof}
Without loss of generality, we can assume that $p$ is a monomial. When $p$ is a monomial, expanding $g.p$ proves the equality.
\end{proof}

We can generalize Lemma \ref{onewedge} using the following lemma.

\begin{lemma}
\label{twowedge}
Suppose $r\geq 2$. Let $p_{1}, p_{2}\in \textrm{}^{r}S_{d}=A_{d, 1}^{r}$. For any $g\in \textup{GL}_{r}(k)$,
\begin{displaymath}
(g.p_{1}\wedge p_{2})_{x_{1}^{d-j_{1}}x_{2}^{j_{1}}\wedge x_{1}^{d-j_{2}}x_{2}^{j_{2}}}=\bigg\lvert\begin{array}{cc}
(g.p_{1})_{x_{1}^{d-j_{1}}x_{2}^{j_{1}}} & (g.p_{1})_{x_{1}^{d-j_{2}}x_{2}^{j_{2}}}\\
(g.p_{2})_{x_{1}^{d-j_{1}}x_{2}^{j_{1}}} & (g.p_{2})_{x_{1}^{d-j_{2}}x_{2}^{j_{2}}}
\end{array}
\bigg\rvert
\end{displaymath}
for all $1\leq j_{1} < j_{2} \leq d$.
\end{lemma}
\begin{proof}
It can be derived from the definition.
\end{proof}

In Lemma~\ref{onewedge}, we see that taking $(g.\star)_{m}$ of $p$ {\it separates} each monomial with respect to the degrees of each variables of $p$ and $m$. Our construction would make use of this {\it phenomenon}. That is, we will {\it control} the degree of one variable, say $x_{r+1}$.

Fix $d$. Let $F$ be a sequence $\{F_{i}\}_{i=0}^{2l-1}\in (\textrm{}^{r}S_{d})^{2l}\subset (\textrm{}^{r+1}S_{d})^{2l}$. Let's define $v_{d}^{r}(F)\in A_{2l+d, 2}^{r+1}$ as follows.
\begin{displaymath}
v_{d}^{r}(F)=\left\{\sum_{i=0}^{2l-1}x_{r+1}^{i}x_{1}^{2l-i}F_{i}\right\}\wedge \left\{\sum_{i=0}^{2l-1}x_{r+1}^{i+1}x_{1}^{2l-i-1}F_{i}\right\}
\end{displaymath}
Note that $[v_{d}^{r}(F)]\in \textup{Hilb}^{P}(\mathbb{P}_{k}^{r})$ where
\begin{displaymath}
P(t)=\binom{r+t}{r}-\binom{r+t-2l-d+1}{r}+\binom{r+t-2l-d-1}{r-2}.
\end{displaymath}
Indeed, the graded ideal
\begin{displaymath}
I_{F}=\bigg\langle \sum_{i=0}^{2l-1}x_{r+1}^{i}x_{1}^{2l-i}F_{i}, \sum_{i=0}^{2l-1}x_{r+1}^{i+1}x_{1}^{2l-i-1}F_{i} \bigg\rangle
\end{displaymath}
of $^{r+1}S$ satisfies the following properties.
\begin{lemma}
\label{hpoint}
$^{r+1}S/I_{F}$ has the Hilbert polynomial
\begin{displaymath}
P(t)=\binom{r+t}{r}-\binom{r+t-2l-d+1}{r}+\binom{r+t-2l-d-1}{r-2}.
\end{displaymath}
Also, $g_{P}=2l+d$ so that $i_{r+1, P, 2l+d}(I_{F})= [v_{d}^{r}(F)]$. If $r\geq 2$ then $I_{F}$ is saturated.
\end{lemma}
\begin{proof}
$I=I_{F}$ is isomorphic to $\langle x_{1}, x_{r+1}\rangle(-2l-d+1)$ as a graded $^{r+1}S$ module. Thus, $\dim_{k} (I_{F})_{t+2l+d}$ is equal to the number of monomials in $^{r+1}S_{t+1}$ which is divisible by $x_{1}$ or $x_{r+1}$, for every $t\geq 0$. This implies that $I_{F}$ has the Hilbert polynomial
\begin{displaymath}
Q(t)=\binom{r+t-2l-d+1}{r}-\binom{r+t-2l-d-1}{r-2}.
\end{displaymath}
$Q$ admits the Macaulay representation
\begin{displaymath}
Q(t)=\binom{r+t-2l-d}{r}+\binom{r+t-2l-d-1}{r-1}.
\end{displaymath}
By the definition of $n(Q)$ in \cite[p. 65]{Gotzmann}, $g_{P}=2l+d$. The regularity of $I_{F}$ is equal to the regularity of $\langle x_{1}, x_{r+1}\rangle(-2l-d+1)$, which is equal to $2l+d$. Let $J$ be the saturation of $I_{F}$. The Hilbert polynomial of $J$ is $Q$. This implies that the regularity of $J$ is at most $g_{P}=2l+d$. Therefore, $\dim_{k} J_{t}=Q(t)=\dim_{k} I_{t}$ for all $t\geq 2l+d$ by \cite[(1.2) Satz, (2.9) Lemma]{Gotzmann}. Suppose $r\geq 2$. If there is a homogeneous $q\in J\setminus I$ then $q\in J_{t}$ for some $t<2l+d$. We derive an inequality $2=\dim_{k} I_{2l+d}=\dim_{k} J_{2l+d}\geq \dim_{k} \langle q\rangle_{2l+d}\geq r+1\geq 3$, which is false.
\end{proof}

We can analyze the polynomial coefficient of $g.v_{d}^{r}(F)$ as follows:
\begin{lemma}
\label{basic}
$\{f_{a, r, l, F}\}_{a=0}^{l-1}$ is a basis for 
\begin{displaymath}
\textup{span}\{\theta_{r+1}((g.v_{d}^{r}(F))_{x_{1}^{2l+d-a}x_{r+1}^{a}\wedge x_{1}^{d+a}x_{r+1}^{2l-a}})|0\leq a\leq l-1\}
\end{displaymath}
where
\begin{displaymath}
f_{a, r, l, F}=\sum_{0\leq i < j \leq 2l-1} \tilde{F_{i}}\tilde{F_{j}}g_{1 r+1}^{i+j-2l+1}\left[\binom{i}{a}\binom{j}{2l-a-1}+\binom{i}{2l-a-1}\binom{j}{a}\right]
\end{displaymath}
\begin{displaymath}
+\sum_{i=0}^{2l-1}\tilde{F_{i}}^{2}g_{1 r+1}^{2i-2l+1}\binom{i}{a}\binom{i}{2l-a-1}
\end{displaymath}
and
\begin{displaymath}
\tilde{F_{i}}=F_{i}(1, g_{12}, \ldots, g_{1 r}).
\end{displaymath}
\end{lemma}
\begin{proof}
Using Lemma~\ref{onewedge} and Lemma~\ref{twowedge}, we can compute that
\begin{displaymath}
f_{a, r, l, F}-f_{a-1, r, l, F}=\theta_{r+1}((g.v_{d}^{r}(F))_{x_{1}^{2l+d-a}x_{r+1}^{a}\wedge x_{1}^{d+a}x_{r+1}^{2l-a}})
\end{displaymath}
for all $1\leq a\leq l-1$ and
\begin{displaymath}
f_{0, r, l, F}=\theta_{r+1}((g.v_{d}^{r}(F))_{x_{1}^{2l+d}\wedge x_{1}^{d}x_{r+1}^{2l}}).
\end{displaymath}
\end{proof}
Let $\psi$ be a sequence $\{\psi_{i}\}_{i=0}^{l-1}\in (\textrm{}^{r}S_{d})^{l}\subset (\textrm{}^{r+1}S_{d})^{l}$. Let's define a sequence $F_{\psi}\in(\textrm{}^{r}S_{d})^{2l}\subset (\textrm{}^{r+1}S_{d})^{2l}$ as follows:
\begin{itemize}
\item $(F_{\psi})_{i}=0$ for all $l\leq i\leq 2l-2$.
\item $(F_{\psi})_{2l-1}=x_{1}^{d}$.
\item $(F_{\psi})_{i}=\frac{1}{i!}\psi_{i}$ for all $0\leq i\leq l-1$.
\end{itemize}
\begin{lemma}
\label{refinement}
$\pi^{\psi}=\{\pi_{j}^{\psi}\}_{j=0}^{l-1}$ is a basis for
\begin{displaymath}
\textup{span}\{\theta_{r+1}((g.v_{d}^{r}(F_{\psi}))_{x_{1}^{2l+d-a}x_{r+1}^{a}\wedge x_{1}^{d+a}x_{r+1}^{2l-a}})|0\leq a\leq l-1\}
\end{displaymath}
where
\begin{displaymath}
\pi_{j}^{\psi}=\frac{(2l-1)!}{(2l-1-j)!}\left[\sum_{a=0}^{l-1-j}\binom{2l-1-j}{a}(-1)^{a}\right]g_{1 r+1}^{2l-1} + \tilde{\psi_{j}}g_{1 r+1}^{j}
\end{displaymath}
and
\begin{displaymath}
\tilde{\psi_{j}}=\psi_{j}(1, g_{12}, \ldots , g_{1r}).
\end{displaymath}
\end{lemma}
\begin{proof}
By the definition of $f_{a, r, l, F_{\psi}}$,
\begin{displaymath}
a!\binom{2l-1}{a}^{-1}f_{a, r, l, F_{\psi}}=\frac{(2l-1)!}{(2l-1-a)!}g_{1 r+1}^{2l-1}
+\sum_{i=a}^{l-1}\frac{1}{(i-a)!}\tilde{\psi_{i}}g_{1 r+1}^{i}
\end{displaymath}
for all $0 \leq a\leq l-1$. Now
\begin{displaymath}
\sum_{a=j}^{l-1}\frac{(-1)^{a+j}}{(a-j)!}\left[a!\binom{2l-1}{a}^{-1}f_{a, r, l, F_{\psi}}\right]=
(2l-1)!\sum_{a=j}^{l-1}\frac{(-1)^{a+j}}{(a-j)!(2l-1-a)!}g_{1 r+1}^{2l-1}
\end{displaymath}
\begin{displaymath}
+\sum_{a=j}^{l-1}\sum_{i=a}^{l-1}\frac{(-1)^{a+j}}{(a-j)!(i-a)!}\tilde{\psi_{i}}g_{1 r+1}^{i}
\end{displaymath}
\begin{displaymath}
=\frac{(2l-1)!}{(2l-1-j)!}\sum_{a=0}^{l-1-j}\binom{2l-1-j}{a}(-1)^{a}g_{1 r+1}^{2l-1}
\end{displaymath}
\begin{displaymath}
+\sum_{i=j}^{l-1}\sum_{a=0}^{i-j}\frac{1}{(i-j)!}\binom{i-j}{a}(-1)^{a}\tilde{\psi_{i}}g_{1 r+1}^{i}=\pi_{j}^{\psi}.
\end{displaymath}
Clearly $\{\pi_{j}|0\leq j\leq l-1\}$ is a linearly independent set. This proves the lemma.
\end{proof}
$\pi^{\psi}$ has the following property. This property depends on the characteristic of $k$, which is zero in this paper.
\begin{lemma}
\label{nonzero}
The coefficient of $g_{1 r+1}^{2l-1}$ in $\pi^{\psi}_{j}$ is non-zero. That is,
\begin{displaymath}
\sum_{a=0}^{l-1-j}\binom{2l-1-j}{a}(-1)^{a}\neq 0
\end{displaymath}
for any choice of integers $l$ and $j$ satisfying $l\geq 1$ and $0\leq j\leq l-1$.
\end{lemma}
\begin{proof}
Note that
\begin{displaymath}
\binom{2l-1-j}{a}\leq \binom{2l-1-j}{a+1}
\end{displaymath}
for all $a$ satisfying $0\leq a \leq l-1-j$.\\
If $l-1-j$ is even,
\begin{displaymath}
\sum_{a=0}^{l-1-j}\binom{2l-1-j}{a}(-1)^{a}=1+\sum_{a=1}^{\frac{l-1-j}{2}}\binom{2l-1-j}{2a}-\binom{2l-1-j}{2a-1}>0.
\end{displaymath}
Similarly, If $l-1-j$ is odd, we can show that
\begin{displaymath}
\sum_{a=0}^{l-1-j}\binom{2l-1-j}{a}(-1)^{a}<0
\end{displaymath}
because the first term is always strictly smaller than the absolute value of the second term.
\end{proof}
\section{NP-hardness of a problem judging the existence of an upper-triangular coordinate}
Suppose $l\geq 3$, $r\geq 2$ and $p=\{p_{i}\}_{i=0}^{l-3}\in k[x_{2}, \ldots , x_{r}]^{l-2}$. Assume that 
\begin{displaymath}
d\geq \max\{\deg(p_{i})|0\leq i\leq l-3\}
\end{displaymath} 
where $\deg(p_{i})$ means the non-homogeneous degree of $p_{i}$. Let's construct $\psi(p)=\{\psi_{i}(p)\}_{i=0}^{l-1}$.
\begin{itemize}
\item Define the first two terms as follows:
\begin{equation}
\label{firstterm}
\psi_{i}(p)=-\frac{(2l-1)!}{(2l-1-i)!}\left[\sum_{a=0}^{l-1-i}\binom{2l-1-i}{a}(-1)^{a}\right]x_{1}^{d}
\end{equation}
 for $i\in\{0, 1\}$.
\item For $2\leq i\leq l-1$, let
\begin{equation}
\label{otherterms}
\psi_{i}(p)=-\frac{(2l-1)!}{(2l-1-i)!}\left[\sum_{a=0}^{l-1-i}\binom{2l-1-i}{a}(-1)^{a}\right]x_{1}^{d}
\end{equation}
\begin{displaymath}
+x_{1}^{d}p_{i-2}\left(\frac{x_{2}}{x_{1}}, \ldots, \frac{x_{r}}{x_{1}} \right).
\end{displaymath}
\end{itemize}
Now we are ready to prove follows.
\begin{theorem}
\label{main1}
Let $l\geq 3$. There is $g\in U_{r+1}$ satisfying $\chi=\chi_{1}^{2d+2l}\chi_{r+1}^{2l}\notin \Xi_{[g.v_{d}^{r}(F_{\psi (p)})]}$ if and only if the ideal $J$ of $k[x_{2}, \ldots, x_{r}]$ generated by $\{p_{i}|0\leq i\leq l-3\}$ has a solution over $k$.
\end{theorem}
\begin{proof}
By definition, $Z_{[v_{d}^{r}(F_{\psi (p)})], \chi}\cap U_{r+1}$ is the zero set of the ideal 
\begin{displaymath}
I\subset \Gamma(U_{r+1}, \mathcal{O}_{U_{r+1}})=k[\{g_{ij}\}_{1\leq i< j\leq r+1}]
\end{displaymath}
 generated by
\begin{displaymath}
\{\theta_{r+1}((g.v_{d}^{r}(F_{\psi(p)}))_{x_{1}^{2l+d-a}x_{r+1}^{a}\wedge x_{1}^{d+a}x_{r+1}^{2l-a}})|0\leq a\leq l-1\}.
\end{displaymath}
By Lemma~\ref{refinement}, $I$ is generated by
\begin{displaymath}
\{\pi^{\psi(p)}_{i}|0\leq i\leq l-1\}.
\end{displaymath} 
It suffices to show that the zero set of $I$ is non empty if and only if the zero set of $J$ is non empty. If there is an element $\{x_{ij}\}_{1\leq i<j\leq r+1}$ in the zero set of $I$, then $x_{1 r+1}=1$ because $\pi_{i}^{\psi(p)}=0$ for $i\in\{0, 1\}$ if and only if $g_{1 r+1}^{2l-1}=1$ and $g_{1 r+1}^{2l-1}-g_{1 r+1}=0$ by Lemma~\ref{nonzero}. Note that $(x_{12}, \ldots, x_{1r})$ is a solution of the system of equation defined by
\begin{displaymath}
\{\pi^{\psi(p)}_{i}|_{g_{1 r+1}=1}|2\leq i\leq l-1\}=\{p_{i}(g_{12}, \ldots, g_{1r})|0\leq i\leq l-3\}
\end{displaymath}
so that $J$ has non-empty zero set. If there is an element $\{x_{i}\}_{i=2}^{r}$ in the zero set of $J$, $\{z_{ij}\}_{1\leq i<j\leq r+1}$ is in the zero set of $I$ if $z_{1i}=x_{i}$ for all $2\leq i\leq r$ and $z_{1 r+1}=1$.
\end{proof}
Theorem~\ref{main1} implies follows.
\begin{corollary}
For any ideal $I$ of a polynomial ring, there is a Hilbert point $v\in\textup{Hilb}^{P}(\mathbb{P}_{k}^{r})$, a choice of closed immersion $\textup{Hilb}^{P}(\mathbb{P}_{k}^{r})\rightarrow \mathbb{P}(\bigwedge^{Q(d)}S_{d})$ and a character $\chi\in X(T_{r+1})$ such that there is an ideal $J$ of $\Gamma(\textup{GL}_{r+1}(k), \mathcal{O}_{\textup{GL}_{r+1}(k)})$ such that $Z_{v, \chi}$ is the zero locus of $J$ and $J$ is a generalization of $I$.
\end{corollary}

Let's consider some decision problems. Let \textbf{SysAl} be a problem asking if a system of algebraic equations over $\mathbb{Q}$ has a solution over $k$ and \textbf{HC} be a problem asking if a graph has a Hamiltonian cycle. Using the proof of Corollary 2.3.2 in \cite[p. 21]{Mar}, we can prove that \textbf{HC} can be reduced to \textbf{SysAl} in polynomial time. By Theorem 10.23 of \cite{Hopcroft}, \textbf{HC} is an NP-complete problem so that \textbf{SysAl} is an NP-hard problem. Let's describe a solvability check problem \textbf{SC} as follows:
\begin{itemize}
\item Given : A rational Hilbert point $v\in\textup{Hilb}^{P}(\mathbb{P}_{k}^{r-1})$, a choice of closed immersion $\textup{Hilb}^{P}(\mathbb{P}_{k}^{r-1})\rightarrow \mathbb{P}(\bigwedge^{Q(d)}S_{d})$ and a character $\chi\in X(T_{r})$.
\item Decide : Is there a coordinate $g\in U_{r}$ satisfying $\chi\notin\Xi_{g.v}$?
\end{itemize}
Here, $v\in \textup{Hilb}^{P}(\mathbb{P}_{k}^{r-1})$ is rational if it represents a saturated homogeneous ideal of $\textrm{}^{r}S$ generated by rational polynomials. 
Theorem~\ref{main1} shows that there is a polynomial time reduction from \textbf{SysAl} to \textbf{SC}. That is,
\begin{corollary}
The problem $\mathbf{SC}$ is NP-hard.
\end{corollary}
There is an extended version of \textbf{SC}, which would be called \textbf{ESC}, described as follows:
\begin{itemize}
\item Given : A rational Hilbert point $v\in\textup{Hilb}^{P}(\mathbb{P}_{k}^{r-1})$, a choice of closed immersion $\textup{Hilb}^{P}(\mathbb{P}_{k}^{r-1})\rightarrow \mathbb{P}(\bigwedge^{Q(d)}S_{d})$ and a finite set of characters $C \subset X(T_{r})$.
\item Decide : Is there a coordinate $g\in U_{r}$ satisfying $C \cap \Xi_{g.v}=\emptyset$?
\end{itemize} 
\textbf{SC} can be reduced to \textbf{ESC} in polynomial time so that we can prove follows:
\begin{corollary}
The problem $\mathbf{ESC}$ is NP-hard.
\end{corollary}
On the other hand, we can use Buchberger's algorithm in \cite{Coxbook} to solve the problem \textbf{ESC} because the zero set of an ideal $I\subset\textrm{} ^{r}S$ is non-empty if and only if $1\notin I$ if and only if the Gr\"obner basis of $I$ respect to the lexicographic(or graded reverse-lexicographic) monomial order contains $1$.

Let's construct an example. Fix natural numbers $r$ and $d$. Suppose $l=3$, $p_{0}\in k[x_{2}, \ldots, x_{r}]$ and $\deg(p_{0})\leq d$. In this case, $p$ is a sequence of length $1$ and the ideal generated by $\{p_{i}|0\leq i\leq l-3\}$ has empty zero locus if and only if $p_{0}$ is a non-zero constant polynomial. Let 
\begin{displaymath}
F'=-6x_{1}^{d+5}+15x_{1}^{d+4}x_{r+1}-10x_{1}^{d+3}x_{r+1}^{2}+x_{1}^{d}x_{r+1}^{5}+\frac{x_{1}^{d+3}x_{r+1}^{2}}{2}p_{0}\left(\frac{x_{2}}{x_{1}}, \ldots, \frac{x_{r}}{x_{1}} \right).
\end{displaymath}
By the definition, $I_{F_{\psi(p)}}=\langle x_{1}F', x_{r+1}F'\rangle$. This means that there is a $g\in U_{r+1}$ such that $\chi_{1}^{2d+6}\chi_{r+1}^{6}\notin \Xi_{[g.v_{d}^{r}(F_{\psi(p)})]}$ if and only if $p_{0}$ is the zero polynomial or $\deg(p_{0})\geq 1$.
\section{A relation between the problem ESC and GIT-semistability}
In this section, every GIT problem is related to the action of $\textup{GL}_{r}(k)$ on $E_{d, b}^{r}$. It will be proved that we can decide whether a rational Hilbert point is GIT-semistable by solving finitely many \textbf{ESC}.  As a consequence of \cite[Criterion 3.3]{Ian}, we have the following lemma.
\begin{lemma}
\label{stability}
A rational point $v\in E_{d, b}^{r}$ is GIT-semistable if and only if $\xi_{d, b}^{r}\in\Delta_{g.v}$ for all $g\in \textup{GL}_{r}(k)$.
\end{lemma}
\begin{proof}
$v$ is semi-stable if and only if it is semi-stable under the action of every maximal torus of $\textup{GL}_{r}(k)$ by \cite[Theorem 2.1]{GIT}. Since every two maximal tori are conjugate, \cite[Criterion 3.3]{Ian} proves the lemma.
\end{proof}
A point in $X(T)_{\mathbb{R}}$ is not in a polytope $\Delta$ if and only if there is a separating hyperplane in $X(T)_{\mathbb{R}}$. That is,
\begin{lemma}
\label{hplane}
For any $g\in\textup{GL}_{r}(k)$ and $v\in E_{d, b}^{r}$, $\xi_{d, b}^{r}\notin \Delta_{g.v}$ if and only if there is an $\omega \in X(T)_{\mathbb{R}}^{\vee}$ such that
\begin{equation}
\label{unstable}
\omega (\xi_{d, b}^{r})<\min \omega (\Xi_{g.v}\otimes_{\mathbb{R}} 1).
\end{equation}
\end{lemma}
For some special choices of $\omega\in  X(T)_{\mathbb{R}}^{\vee}$ and $v$, we can still guarantee \eqref{unstable} for every $g\in L_{r}$.
\begin{lemma}
\label{lower}
Suppose there are $v\in E_{d, b}^{r}$ and $\omega \in X(T)_{\mathbb{R}}^{\vee}$ satisfying
\begin{displaymath}
\omega (\xi_{d, b}^{r})<\min \omega (\Xi_{v}\otimes_{\mathbb{R}} 1)
\end{displaymath}
and $\omega (\chi_{i})\leq \omega (\chi_{i+1})$ for all $1\leq i < r$. Then, for any $l\in L_{r}$,
\begin{displaymath}
\omega (\xi_{d, b}^{r})<\min \omega (\Xi_{l.v}\otimes_{\mathbb{R}} 1).
\end{displaymath}
\end{lemma}
\begin{proof}
Suppose $\eta\in \Xi_{l.v}\otimes_{\mathbb{R}} 1\setminus \Xi_{v}\otimes_{\mathbb{R}} 1$. It suffices to show that $\omega(\eta)\geq \min \omega (\Xi_{v}\otimes_{\mathbb{R}} 1)$. By definition, there is an $m\in W_{d,b}^{r}$ satisfying $\eta\in\Xi_{l.m}$ and $\Xi_{m}\subset \Xi_{v}$. By expanding $l.m$, we can prove that
\begin{displaymath}
\omega(\eta)\geq \min\omega(\Xi_{m}\otimes_{\mathbb{R}} 1)
\end{displaymath}
using the condition $\omega (\chi_{i})\leq \omega (\chi_{i+1}), \forall 1\leq i\leq r-1$. Since $\Xi_{m}\subset \Xi_{v}$, we can deduce that $\min\omega(\Xi_{m}\otimes_{\mathbb{R}} 1)\geq \min\omega(\Xi_{v}\otimes_{\mathbb{R}} 1)$. Thus the claimed statement is true.
\end{proof}
Now, we can restate the condition for $v$ to be unstable.
\begin{theorem}
\label{main2}
Suppose $v\in E_{d, b}^{r}$. $v$ is unstable if and only if there are $u\in U_{r}$ and $q\in\Sigma_{r}$ satisfying
\begin{displaymath}
\xi_{d, b}^{r}\notin \Delta_{uq.v}
\end{displaymath}
\end{theorem}
\begin{proof}
If part is obvious by Lemma~\ref{stability}. Suppose there is $g\in\textup{GL}_{r}(k)$ satisfying
\begin{displaymath}
\xi_{d, b}^{r}\notin \Delta_{g.v}.
\end{displaymath}
By Lemma~\ref{hplane}, there is $\omega\in X(T)_{\mathbb{R}}^{\vee}$ satisfying
\begin{displaymath}
\omega (\xi_{d, b}^{r})<\min\omega (\Xi_{g.v}\otimes_{\mathbb{R}} 1).
\end{displaymath}
There is a $p\in \Sigma_{r}$ satisfying $\omega (\chi_{p(i)})\leq \omega (\chi_{p(i+1)})$ for all $i$. Let's define $\omega_{p}(\chi_{i})=\omega(\chi_{p(i)})$. Then, 
\begin{displaymath}
\omega_{p} (\xi_{d, b}^{r})=\omega (\xi_{d, b}^{r})<\min\omega (\Xi_{g.v}\otimes_{\mathbb{R}} 1)=\min\omega_{p} (\Xi_{p^{-1}g.v}\otimes_{\mathbb{R}} 1).
\end{displaymath}
Now there are $l\in L_{r}$,$u\in U_{r}$ and $q\in \Sigma_{r}$ satisfying $p^{-1}g=luq$ by the LU-decomposition of general non-singular matrix. $p^{-1}g.v$ and $\omega_{p}$ satisfies the condition of Lemma~\ref{lower}. Thus,
\begin{displaymath}
\omega_{p} (\xi_{d, b}^{r})<\min\omega_{p} (\Xi_{l^{-1}luq.v}\otimes_{\mathbb{R}} 1)=\min\omega_{p} (\Xi_{uq.v}\otimes_{\mathbb{R}} 1).
\end{displaymath}
By Lemma~\ref{hplane}, $\xi_{d, b}^{r}\notin \Delta_{uq.v}$, as desired.
\end{proof}
Using Theorem~\ref{main2} and Lemma~\ref{hplane}, we can solve \textbf{ESC} for each choice of $\omega\in X(T)_{\mathbb{R}}^{\vee}$ and $q\in\Sigma_{r}$ to check if
\begin{displaymath}
\{\chi\in X(T)|\omega(\chi)\leq\omega(\xi_{d, Q(d)}^{r})\}\cap \Xi_{uq.v}=\emptyset
\end{displaymath}
for a rational $v\in \textup{Hilb}^{P}(\mathbb{P}_{k}^{r-1})$ and an integer $d\geq g_{P}$. Note that we have to consider finitely many $\omega$'s because $A_{d, b}^{r}$ has only finitely many weights with respect to the action of $T_{r}$. In this way, we can check if $v$ is semistable or not. This fact implies that there is an algorithm deciding if a rational $v\in\textup{Hilb}^{P}(\mathbb{P}_{k}^{r-1})$ is GIT-semistable or not.
\bibliographystyle{plain}
\bibliography{Turing}
\end{document}